\documentclass[12pt,a4paper]{amsart}
\usepackage[english]{babel}

\usepackage{graphicx}

\usepackage{amssymb}
\usepackage{amsmath}

\advance\hoffset-15mm \advance\textwidth30mm
\advance\voffset-15mm \advance\textheight30mm

\usepackage{amscd}

\newcommand{\PP}{\ensuremath{\mathbb{P}}}

\newcommand{\ZZ}{\ensuremath{\mathbb{Z}}}

\def\hol{{\mathcal{O}}}

\DeclareMathOperator{\Aut}{Aut}

\theoremstyle{plain}
\newtheorem{teo}{Theorem}[section]{\bf}{\it}
\newtheorem{conj}[teo]{Conjecture}{\bf}{\it}
\newtheorem{prop}[teo]{Proposition}{\bf}{\it}
\newtheorem{df}[teo]{Definition}{\bf}{\it}
\newtheorem{lem}[teo]{Lemma}{\bf}{\it}
\newtheorem{cor}[teo]{Corollary}{\bf}{\it}

\theoremstyle{definition}
{\bf}{\rm}
{\bf}{\rm}
{\bf}{\rm}
\newtheorem{ex}[teo]{Example}{\bf}{\rm}
{\bf}{\rm}
\newtheorem{oss}[teo]{Remark}{\bf}{\rm}
{\bf}{\rm}
\newtheorem{nota}[teo]{Notations}{\bf}{\rm}

\makeatletter\@addtoreset{equation}{section}\makeatother

\makeatletter\@addtoreset{subsection}{equation}\makeatother

\newcommand\qmatrix[2][1]{\left(\renewcommand\arraystretch{#1}
\begin{array}{*{20}r}#2\end{array}\right)}

\newcommand{\p}{\mathbb{P}}
\let \PP \p

\newcommand{\na}{\mathbb{N}}

\def\AA{{\mathbb A}}

\newcommand{\aut}{\text{Aut}}
\def\SAut{\mathop{\rm SAut}}
\let \saut \SAut
\def\Stab{\mathop{\rm Stab}}
\def\rk{\mathop{\rm rk}}
\def\cone{\mathop{\rm cone}}

\def\Pic{\mathop{\rm Pic}}
\newcommand{\fie}{\mathbf{k}}
\let \kk \fie
\newcommand{\summ}{\displaystyle\sum}
\newcommand{\fra}{\displaystyle\frac}
\newcommand{\km}[1]{(\kk^2 \setminus \{0\})/(\ZZ/#1 \ZZ)}
\newcommand{\TX}{\mathfrak T_X}

\begin{document}
\sloppy

\title[Unirationality and existence of infinitely transitive models]{Unirationality and existence of infinitely transitive models}

\author[F.~Bogomolov, I.~Karzhemanov, K.~Kuyumzhiyan]{Fedor Bogomolov\lowercase{$\,^{a,b}$}, Ilya Karzhemanov\lowercase{$\,^{a}$}, Karine Kuyumzhiyan\lowercase{$\,^{b}$}}
\address{$^{a}\,$Courant Institute of Mathematical Sciences, N.Y.U. \\
 251 Mercer str. \\
 New York, NY 10012, US;}
 \address {$^{b}\,$National Research University Higher School of Economics \\
GU-HSE,7. Vavilova Str. Moscow 117 312, Russia}
\email{bogomolo@cims.nyu.edu} 
\email{karzhema@cims.nyu.edu} 
\email{karina@mccme.ru}

\thanks{The first author was supported by NSF grant DMS-1001662
and by AG Laboratory HSE, RF government
grant, ag.\,11.G34.31.0023. 
The third author
was supported by AG Laboratory HSE, RF government
grant, ag.\,11.G34.31.0023, by the ``EADS Foundation Chair in Mathematics", Russian-French Poncelet Laboratory (UMI 2615 of CNRS), and Dmitry Zimin fund ``Dynasty". }

\begin{abstract}
We study unirational algebraic varieties and the fields of
rational functions on them. We show that after adding a finite
number of variables some of these fields admit an \emph{infinitely
transitive model}. The latter is an algebraic variety with the
given field of rational functions and an infinitely transitive
regular action of a group of algebraic automorphisms generated by unipotent
algebraic subgroups. We expect that this property holds for all
unirational varieties and in fact is a peculiar one for this class
of algebraic varieties among those varieties which are rationally
connected.
\end{abstract}



\maketitle

\section{Introduction}
\label{section:intro}

This article aims to relate unirationality of a given algebraic variety with the property of being
a homogeneous space with respect to unipotent algebraic group action. More precisely,
let $X$ be an algebraic variety defined over a field~$\fie$, and $\aut(X)$ be the group of regular
automorphisms of $X$. 
Let also $\saut(X) \subseteq \aut(X)$ be the subgroup
generated by algebraic groups isomorphic to the additive group $\mathbb{G}_a$.

\begin{df}[cf. \cite{AFKKZ}]
\label{definition:def-1}
We call variety~$X$ \emph{infinitely transitive}
if for any $k\in\na$ and any two collections of points $\{P_1,\ldots,P_k\}$ and $\{Q_1,\ldots,Q_k\}$
on $X$ there exists an element $g\in \saut(X)$ such
that $g(P_i)=Q_i$ for all $i$. Similarly, we call~$X$ \emph{stably infinitely transitive} if $X \times \fie^m$ is infinitely transitive for some~$m$.
\end{df}

Recall that in Birational Geometry adding a number $m$ of
algebraically independent variables to the function field
$\fie(X)$ is referred to as \emph{stabilization}. Geometrically
this precisely corresponds to taking the product $X \times\fie^m$
with the affine space. Note also that if $X$ is infinitely
transitive, then it is unirational, i.e., $\fie(X) \subseteq
\fie(y_1,....y_m)$ for some $\fie$-transcendental elements
$y_i$ (see \cite[Proposition 5.1]{AFKKZ}). This suggests to regard
(stable) infinite transitivity as a birational property of $X$ (in
particular, we will usually assume the test variety $X$ to be
smooth and projective):

\begin{df}
\label{definition:def-2}
We call variety $X$ \emph{stably b-infinitely transitive} if the field $\fie(X)(y_1,....y_m)$
admits an infinitely transitive model (not necessarily smooth or projective) for some $m$ and $\fie(X)$-transcendental elements $y_i$.
If $m=0$, we call $X$ \emph{b-infinitely transitive}.
\end{df}

\begin{ex}
\label{example:in-tr-rat}
The affine space $X:=\fie^{\dim X}$ is stably infinitely transitive (and infinitely transitive when $\dim X \geq 2$), see~\cite{KZ}.
More generally, any rational variety is stably b-infinitely transitive, and it is b-infinitely transitive if the dimension $\geq 2$.
\end{ex}

Example~\ref{example:in-tr-rat} suggests that being stably b-infinitely transitive does not give anything interesting for rational varieties.
In the present article, we put forward the following:

\begin{conj}
\label{conj:c-main}
Any unirational variety $X$ is stably b-infinitely transitive.
\end{conj}

Thus, Conjecture~\ref{conj:c-main} together with the above
mentioned result from \cite[Proposition 5.1]{AFKKZ} provides
a (potential) characterization of unirational varieties among all those which
are rationally connected. Note also that the class of rationally
connected varieties contains all stably b-infinitely transitive
varieties. 
We think that not every rationally connected variety is stably birationally infinitely transitive. 
In particular we expect that generic Fano hypersurfaces 
from the family considered by Kollar in~\cite{kol-hyp} are
not stably birationally infinitely transitive.
These are generic smooth hypersurfaces of degree
$d$ in  $\p^{n+1}$, $d>\fra{2}{3}(n+3)$.
Our expectations are based on the Kollar's fundamental
observation (see~\cite[Thm. (4.3)]{kol-hyp}) which yields strong
restrictions on  any surjective map
of a uniruled variety of the same dimension
on such a hypersurface.

\begin{oss}
Originally, the study of infinitely transitive varieties was
initiated in the paper \cite{KZ}. We also remark one application
of these varieties to the L\"uroth problem in \cite{AFKKZ}, where
a non-rational infinitely transitive variety was constructed. See~\cite{freu} for the properties of locally nilpotent derivations (LNDs for short), \cite{Po2} for the Makar-Limanov invariant, and 
\cite{AKZ}, \cite{dan-giz}, \cite{hui-man}, \cite{KPZ}, \cite{KM}, and \cite{Per} 
for other results, properties and applications of
infinitely transitive (and related) varieties.
\end{oss}

We are going to verify Conjecture~\ref{conj:c-main} for some
particular cases of $X$ (see Theorems~\ref{theorem:t-a-1},
\ref{theorem:t-g-1} and Propositions~\ref{theorem:cubic}, \ref{theorem:sing-quart} and
\ref{theorem:3-quadrics-int} below). At this stage, one should note 
that it is not possible to lose the stabilization assumption in
Conjecture~\ref{conj:c-main}:

\begin{ex}
\label{example:counter-ex} Any three-dimensional algebraic variety
$X$ with an infinitely transitive model is rational. Indeed, let
us take a one-dimensional algebraic subgroup $G\subset\saut(X)$ acting on $X$ with a free
orbit. Then $X$ is birationally isomorphic to $G \times Y$ (see
Remark~\ref{remark:converse} below), where $Y$ is a rational
surface (since $X$ is unirational). On the other hand, if $X :=
X_3 \subset \p^4$ is a smooth cubic hypersurface, then it is
unirational but not rational (see \cite{cl-gr}). However,
Conjecture~\ref{conj:c-main} is true as stated for $X_3$, because
$X_3$ is stably b-infinitely transitive (see
Proposition~\ref{theorem:cubic} below). In this context, it would
be also interesting to settle down the case of the quartic
hypersurface $X_4$ in $\p^4$ (or, more generally, in $\p^n$ for
arbitrary $n$), which relates our subject to the old problem of
(non-)unirationality of (generic) $X_4$ (cf. 
Remark~\ref{remark:other-ci} below).
\end{ex}

\smallskip

\begin{nota}
Throughout the paper we keep up with the following:

\begin{itemize}

\item $\fie$ is an algebraically closed field of
characteristic zero and $\fie^\times$ is the multiplicative group of $\fie$;

\smallskip

\item $X_1
\approx X_2$ denotes birational equivalence between two algebraic
varieties $X_1$ and $X_2$;

\smallskip

\item we abbreviate infinite transitivity (transitive,
transitively, etc.) to inf.~trans.

\smallskip

\end{itemize}

\end{nota}

\noindent
{\bf Acknowledgment.} The second author would like to thank Courant Institute
for hospitality. The second author has also benefited from discussions with I.\,Cheltsov, Yu.\,Prokhorov, V.\,Shokurov, and K.\,Shramov. The third author would like to thank I. Arzhantsev for fruitful discussions. The authors are grateful to the referee for valuable comments. 
The authors thank the organizers of the summer school and the conference in Yekaterinburg (2011),
where the work on the article originated. The result was presented at the Simons Symposium ``Geometry Over Non-Closed
Fields'' in February 2012. The first author wants to thank Simons Foundation for support and participants
of the Symposium for useful discussions. In particular, the question of J.-L. Colliot-Th\'el\`ene (see below) was raised during the Symposium.

\smallskip

\section{Varieties with many cancellations}
\label{section:m-c}

\refstepcounter{equation}
\subsection{The set-up}

The goal of the present section is to prove the following:

\begin{teo}
\label{theorem:t-a-1} Let $K := \fie(X)$ for some (smooth
projective) algebraic variety $X$ of dimension $n$ over $\fie$. We
assume there are $n$ presentations (we call them
\emph{cancellations} (of $K$ or $X$)) $K = K'(x_i)$ for some
$K'$-transcendental elements~$x_i$, algebraically independent over~$\fie$. Then there exists an
inf. trans. model of $K(y_1,\ldots,y_n)$ for some
$K$-transcendental elements $y_i$.
\end{teo}

Let us put Theorem~\ref{theorem:t-a-1} into a geometric
perspective. Namely, the presentation $K = K'(x_i)$ reads as there
exists a model of $K$, say $X_i^n$, with a surjective
regular map $\pi_i : X_i^n \to Y_i^{n-1}$ and general fiber
$\simeq\p^1$ such that $\pi_i$ admits a section over an open subset in
$Y_i^{n-1}$. Moreover, by resolving indeterminacies, we may assume
$X_i^n := X$ fixed for all $i$. Then, since $K$ admits $n$ cancellations,
$n$ vectors, each  tangent to a fiber of some $\pi_i$, span the tangent space
to $X$ at the general point. Indeed, 
we have a map to $\p^n$
$$
X \dasharrow \p^n, \quad x\mapsto (1:x_1(x):\ldots:x_n(x)).
$$
It is dominant since elements  $x_1,\ldots,x_n$ are algebraically independent over~$\fie$, and the tangent map is surjective at the general point. 
So we obtain the geometric counterpart of Theorem~\ref{theorem:t-a-1}:

\begin{teo}
\label{theorem:t-g-1} 
Let $X$ be a smooth projective variety of
dimension $n$. Assume that there exist~$n$ morphisms $\pi_i :
X \to Y_i$ satisfying the following:

\begin{enumerate}

\item\label{c-1} $Y_i$ is a (normal) projective variety such that $\pi_i$ admits a
section over an open subset in $Y_i$;

\smallskip

\item\label{c-2} for the general point $\zeta \in X$ and the fiber $F_i=\p^1_i := \pi_i^{-1}(\pi_i(\zeta)) \simeq
\p^1$, vector fields $T_{F_1,\,\zeta}, \ldots, T_{F_n,\,\zeta}$
span the tangent space $T_{X,\,\zeta}$.

\end{enumerate}
Then $X$ is stably b-inf. trans.
\end{teo}

Note that existence of a section over an open subset on~$Y_i$ means (almost by definition) birational triviality of the fibration~$\pi_i$.


In Sections~\ref{sec:one-dim} and~\ref{sec:two-dim} we illustrate our arguments by considering the cases 
when $\dim X = 1$ and~$2$, respectively.
In higher dimensions we additionally need the following:
\begin{enumerate}
\setcounter{enumi}2

\item\label{i-2} 
for some ample line bundles $H_i$ on $Y_i$ and their pullbacks $\pi_i^*H_i$ to $X$,
the $n\times n$-matrix~(\ref{Matrix}) $(\pi_i^*H_i\cdot \p^1_i)$ is of maximal rank (in particular, the classes
of $\pi_1^*H_1,\ldots,\pi_n^*H_n$ in $\Pic(X)$ are linearly independent).
\end{enumerate}
In particular, this means that the fibers $\p^1_1$, $\p^1_2, \ldots$, $\p^1_n$ are linearly independent in $H_2(X)$. In Sections~\ref{sec:constr-TX}, \ref{sec:good-strata} and~\ref{sec:together} we prove Theorem~\ref{theorem:t-g-1}, assuming that the condition~\ref{i-2} is satisfied.
Furthermore, adding new variables (i.e., forming the product of~$X$ and an affine space) and passing to a (good) birational model, we may assume that~\ref{i-2} holds, see Sections~\ref{incr-t} and ~\ref{sec:incr-rank}.

\refstepcounter{equation}
\subsection{One-dimensional case} \label{sec:one-dim} 
Variety $\hol(m)^\times_{\p^1}$ (or, equivalently, $\hol(-m)^\times_{\p^1}$) is just an affine cone minus the origin over a rational normal curve of degree~$m$. 
Thus $\hol(m)^\times_{\p^1}$ is a quasiaffine toric variety, so it is infinitely transitive by~\cite[Theorem 0.2(3)]{AKZ}. Indeed, we can use only those automorphisms which preserve the origin, i.e., for $m$-transitivity on $\kk^2\setminus \{0\}$ we use $(m+1)$-transitivity on $\kk^2$. 

\refstepcounter{equation}
\subsection{Two-dimensional case} \label{sec:two-dim} 
Let us study now the next simplest case when $X=\p^1 \times \p^1$. Choose $H_2:=\hol(1)$ on the first factor $\p^1$ and, similarly, $H_1:=\hol(1)$ on the second factor $\p^1$. 
Now take the pullbacks $\pi_1^* H_1$ and $\pi_2^*H_2$ to $X$ and throw away their zero sections. We obtain a toric bundle over~$X$ 
isomorphic 
to $(\kk^2\setminus \{0\})\times (\kk^2\setminus \{0\})$. The latter is inf. trans. since $\kk^2\setminus \{0\}$ is (cf. the one-dimensional case above). 

More generally, if one starts with $H_2=\hol(m_2)$ and $H_1=\hol(m_1)$ for some $m_i\geqslant 1$, then 
the resulting variety will be $(\km{m_1})\times (\km{m_2})$. It is again inf-transitive being the product of two inf-transitive varieties. Indeed, for $m_i>1$ the corresponding variety is just the smooth locus on the corresponding toric variety, and its inf-transitivity is shown in~\cite[Theorem 0.2(3)]{AKZ}.

\begin{oss}
The product of two (quasiaffine or affine) inf-transitive varieties is inf-transitive. Indeed, 
we call variety~$X$ {\it flexible} if the tangent space at every smooth point on~$X$ is generated by the tangent vectors to the orbits of one-parameter unipotent subgroups in~$\Aut(X)$. It was shown in~\cite{AFKKZ} that for affine~$X$ 
being flexible is equivalent to inf-transitivity. But clearly the product of two flexible varieties is again flexible. 
\end{oss}

\refstepcounter{equation}
\subsection{Construction of an inf. transitive model in the simplest case} \label{sec:constr-TX}

Recall the setting. In the notation of Theorem~\ref{theorem:t-g-1}, we choose very ample line bundles $H_i$ on each~$Y_i$, $i=1,\ldots, n$, take their pullbacks $\pi_i^*H_i$ to $X$, put $m_{ij}:=(\pi_i^*H_i)|_{\p^1_j}$, 
and form the {\it intersection matrix}
\begin{equation}\label{Matrix}
M_n=M_n(X)=(m_{ij})_{1\leqslant i,j \leqslant n}, \quad m_{ij}=(\pi_i^*H_i)|_{\p^1_j} \,.
\end{equation}
Clearly, for all~$i$ we have $(\pi_i^*H_i)|_{\p^1_i}=0$; however, for $i \neq j$, $(\pi_i^*H_i)|_{\p^1_j}>0$, being equal the restriction of $H_i$ to an image of a generic $\p^1_j$ via $p_i$ (i.e. the restriction of a bundle on the variety~$Y_i$). 
The matrix~$M_n$ defines a linear map from a subgroup of the Pickard group $\Pic X$ to $\ZZ^n$. 
In this section we suppose that the classes of $\p^1_1, \ldots$, $\p^1_n$ in $H_2(X)$ are linearly independent, and also that $\det M_n\neq 0$. 
Our goal is to construct a quasiaffine variety $\TX$, $\TX \approx X \times k^N$ for some $N$, equipped with a collection of projections to quasiaffine varieties 
$\bar Y_i$ with generic fibers being equal to $\km{m}$, and such that an open subset  
of $\TX$ is inf-transitive, cf. Section~\ref{sec:two-dim}. The existence of a good open subset will be shown in Section~\ref{sec:good-strata}.

To start with, let us set 
$$
\bar Y_i := \mbox{the affine cone }\hol_{Y_i}(H_i)^{\times}\mbox{ minus the origin}
$$ 
over $Y_i$ embedded via $H_i$, $1\leqslant i \leqslant n$. It is a quasiaffine variety.

\subsubsection{Technical step -- adding one more coordinate}
We already embedded $Y_i$ into affine varieties, now we also need to embed~$X$. 
For this purpose, we take a very ample line bundle~$H_0$ on~$X$, replace $X$ with $X'=X\times \PP^1$, and $Y_i$ with $Y'_i=Y_i\times \PP^1$. Let also $Y_0=X$, clearly we have $X'\to Y_0=X$, which makes the situation absolutely  symmetric with respect to indices $0,1,\ldots,n$. We modify the set of $H_i$s in the following way: for every $i>0$, we construct $H_i'$ on $Y_i'$ being the sum of the trivial lift of $H_i$ from $Y_i$ and $\hol(1)$ on the new $\p^1$ (in fact here we can take any $\hol(n_i)$). Now the intersection matrix $M_{n+1}(X')$ takes the form
$$
M_{n+1}(X')=\qmatrix{0 & k_1 & \ldots & k_n \\1& & & \\ \vdots & & M_n  & \\1 & & & \\} . 
$$
Here $k_i := H_0 \cdot \p^1_i$. We further denote $X'$ just by $X$ and $n+1=\dim X'$ just by $n$, keeping in mind that one of our projections is just a trivial projection. We also assume that one column of our matrix contains only $1$s (and one $0$ on the diagonal).

\subsubsection{The construction of $\TX$}

We construct a vector bundle
$$
H_1\times_X \pi_2^*H_2\times_X \pi_3^*H_3\times_X  \ldots \times_X \pi_n^*H_n
$$
and furthermore a toric bundle 
\begin{equation}\label{T_X}
\mathfrak{T}_X=(H_1)^\times \times_X (\pi_2^*H_2)^\times \times_X  \ldots \times_X (\pi_n^*H_n)^\times.
\end{equation}
We denote by~$\delta$ the canonical projection $\TX\to X$. Line operations in the intersection matrix~(\ref{Matrix}) correspond to base changes in this toric bundle (Neron-Severi torus). 
For our convenience, we fix below the following set of line bundles $L_1, \ldots, L_n \in \langle H_1,  \pi_2^*H_2, \ldots, \pi_n^*H_n \rangle$:

(i) each of them should be primitive in the lattice $\ZZ(H_1,  \pi_2^*H_2, \ldots, \pi_n^*H_n)$, 

(ii) all in total, they should be linearly independent in the lattice 
$$
\ZZ(H_1,  \pi_2^*H_2, \ldots, \pi_n^*H_n).
$$

They can be chosen in the following way. There is a map
$$
\begin{CD} 
\ZZ(H_1,\ldots,H_n) @>M_n>> \ZZ^n @>i\mbox{th coordinate}>>\ZZ . 
 \end{CD} 
$$
Its kernel has dimension $n-1$, and $H_i$ itself belongs to the kernel. So in fact there is a map 
$\ZZ(H_1, \ldots, \check H_i,\ldots,H_n)\to \ZZ$, and $L_i$ is any covector defining this map. All in total, they can be chosen linearly independent. 

\subsubsection{Construction of local 2-dimensional coordinates}
Recall that we denote by $\bar Y_i$ the total space of $H_i^\times \to Y_i$.

\begin{lem} \label{projection}
For each $i$, there is a fibration $\varphi_i: \TX\to \bar Y_i$ such that its general fiber equals  
$(\km{m_i})\times T^{n-2}_i$, where $T^{n-2}_i\simeq (\kk^\times)^{n-2}$. 
\end{lem}
\begin{proof}
Choose a basis $\langle H_i, L_i, H'_1,\ldots,H'_{n-2} \rangle$ and a linear map 
$\ZZ^n \to \ZZ^2$ which is just taking the two first coordinates in the new basis. 
Its kernel will correspond precisely to a $(n-2)$-dimensional torus, the bundle $H_i$ will provide us with the affine cone $\bar Y_i$ over $Y_i$, 
and the bundle $L_i$ restricted to $\p^1_i$ will form a quasiaffine fiber of form $\km{m}$ over a general point of $\bar Y_i$. 
\end{proof}

We have a commutative diagram.

\begin{equation}
\begin{CD} 
\TX	@>T_i^{n-2}>>	L_i^{\times} \times H_i^{\times}\\ 
@VVV	@VV{\km{n_i}}\times \ldots V\\
 X	@>>>	Y_i 
 \end{CD} 
 \end{equation}
 \medskip
 \begin{equation}\label{Diagr}
\text{where }\;
\begin{CD} 
 \TX @>T_i^{n-2}\times L_1^\times \times \p^1>> \bar Y_i @>>>Y_i 
 \end{CD} 
\end{equation}
This realization will be intensively used below. Note that the fibration is trivial over any open subset $U$ in $Y_i$ such that all the fibers of $\pi_i$ are $\mathbb P^1$s 
over $U$ and the restriction of all $H$s are generic on these fibers, 
and respectively over $\bar Y_i$. So if one fixes a finite number of points $P_1,\ldots,P_s$ in $\bar Y_i$, 
we can choose an open subset $U'$ in $\bar Y_i$ containing the fibers passing through all these points (since it is quasiaffine). 

\begin{lem}\label{loccoord}
At the general point $x$ on $\TX$, local coordinates on $(\km{m_i}$-fibers from Lemma~\ref{projection}, 
$i=1,\ldots, n$, form a system of local coordinates on $\TX$ at~$x$.
\end{lem}
\begin{proof}
In the notation of Lemma~\ref{projection}, tangent space to each fiber of $\varphi_i$ is spanned by a pair of the tangent vectors to $\km{m_i}$ and by tangent vectors to $T^{n-2}_i$. By 
the condition~\ref{c-2} of Theorem~\ref{theorem:t-g-1} and by non-degeneracy of matrix~$M_n$, the tangent vectors to $\km{m_i}$ are linearly independent, which proves the assertion. 
\end{proof}

\subsubsection{Quasiaffineness of $\TX$}
Here we exploit the projections~(\ref{Diagr}) and the technique from the proof of Lemma~\ref{projection}.

\begin{lem}
The variety $\TX$ is quasiaffine.
\end{lem}
\begin{proof}
The bundle $H_1$ gives an embedding of $X$ to a projective space $\p^{N_1}$, and every $H_i$, $i=2,\ldots,n$, embeds $Y_i$ to a $\p^{N_i}$. 
The variety $\TX$ is now $\{(x,l_2,\ldots,l_n)\}$ such that $x\in X$, $l_i\in \cone (\pi_i(x))$ in $\AA^{N_1+N_2+\ldots+N_n}$. 
\end{proof}


Note that $\mathfrak{T}_X \to X$ is a principal toric
bundle which has a section (the diagonal), and all the fibers are isomorphic to $(\fie^\times)^{n}$ (see formula~(\ref{T_X})). 
In particular, we have 
 $\mathfrak{T}_X \approx X\times \fie^{n}$.

\subsubsection{Idea of further proof}

\begin{prop}
\label{theorem:p-t-g-1} The variety $\mathfrak{T}_X$ is stably b-inf. trans.
\end{prop}

Its proof will be given in Section~\ref{sec:together}. We use the ideas from \cite{KZ}, \cite{AKZ} and~\cite{AFKKZ} 
to move an $m$-tuple of general (in the sence of Section~\ref{sec:good-strata}) points to another such $m$-tuple.

\refstepcounter{equation}
\subsection{Stratification on~$X$}
Let $q\in X$ be an arbitrary point. 
We denote by $X(q)$ the locus of all points on~$X$ connected to~$q$ by a sequence of smooth fibers $\p^1_i$ 
of the projections $\pi_i$, $1\leqslant i \leqslant n$.

\begin{lem} \label{dim-stab-subvar}
Let $Z$ be an irreducible subvariety of $X$.
Consider all smooth fibers  $\PP^1_i$ passing
through the points of $Z$ and the union $Z'$ of all such fibers.
Then either $\dim Z' > \dim Z$
or all smooth fibers  $\p^1_i$ which contain points
in $Z$ are actually contained in the closure $\bar Z$.
\end{lem}

\begin{proof} 
If the curve $\PP^1_i$ intersects $Z$ but is not contained
in $Z$ then the curves in the same family intersect
an open subvariety in $Z$ since the subvariety
$\tilde X_i$ consisting of curves $\p^1_i$ is an open
subvariety of $X$. Hence in the latter case
$\dim Z' > Z$. Otherwise all the smooth fibers  $\PP^1_i$ which contain points
in $Z$ are actually contained in the closure of $Z$.
Note that the same holds even if a line $\PP^1_i$ intersects the closure
$\bar Z$ but is not contained in $\bar Z$.
\end{proof}

\begin{cor} \label{cor:chain}
Every point in $X(q)$ is connected
to $q$ by a chain of $\PP^1_i$ of length at most~$n^2$.
\end{cor}
\begin{proof} 
Indeed, let $X_p(q)$ be a subvariety obtained after adding the points
connected by the chains of curves of length at most $p$. It is a union of algebraic subvarieties of $X$
of dimension $\leqslant p$.

Then by adding the curves from all $n$ families of $\PP^1_i$
we either increase the dimension of every component
of maximal dimension, or one of them $X_p^0(q)$ is invariant, i.e.
all smooth fibers $\p^1_i$ which contain points
in $X_p^0(q)$ are actually contained in the closure of $X_p^0(q)$.
Note that in the latter case  since $q\in X_p^0(q)$, all other components are contained in $X_p^0(q)$, 
and hence $X_p^0(q)= X(q)$. 
Thus after adding lines from different families we obtain
either $X(q)$ or a variety $X_{p+n}(q)$ with maximal component
of greater dimension. Thus we will need at most
$n^2$ lines to get $X(q)$.
\end{proof} 

\begin{oss}
If we started with a generic point $q\in X$, then it follows from Lemma~\ref{dim-stab-subvar} that $\dim X(q)=n$. 
Indeed, the condition~\ref{c-2} of Theorem~\ref{theorem:t-g-1} implies that the tangent vectors to the smooth $\p^1_i$-fibers in $q$ generate the full 
tangent space in~$q$, and if $X(q)$ was of lower dimension then the tangent space would also be of lower dimension. 
\end{oss}

\begin{oss} 
The bound in Corollary~\ref{cor:chain} is not effective. By a more thorough examination one can show that the sequence 
$$
X_0(q) \subseteq X_1(q) \subseteq \ldots
$$
stabilizes earlier than at the $n^2$th step.  
\end{oss}

\begin{cor} We can apply the same in reverse. Consider $x\in X(q)$. 
Then all points in $X$ which are connected to $x\in X(q)$ can be connected
by a chain of length at most $n^2 +n$.
\end{cor}
\begin{proof}
Indeed, for any such  point $x'$
we have $X_{n^2}(x')$ of dimension $n$ and hence contains an open
subvariety in $X$. It may take at most $n$ $\p^1_i$s to connect
to~$x$.
\end{proof} 

Thus the variety obtained from general points in $X$ in $n^2$ steps coincides with the subvariety of all points in $X$ connected to
general point by a chain of smooth lines. 

\refstepcounter{equation}
\subsection{Construction of a big open subset in $\mathfrak{T}_X$}\label{sec:good-strata} 

Now we pass from stratification on $X$ to stratification on $\TX$. 
We stratify~$\TX$ in the following way: we take toric preimages for every strata in~$X$. 
For our needs we take the toric preimage of $X(q)$ for a general point $q\in X$. Note that for every fiber $\p^1_i$ of $\pi_i$ 
its preimage is $((\km{m})\times T^{n-2}$, and for every chain $P_1 - P_2 -\ldots - P_k$ connecting two points in $X$ there is a chain 
$\bar P_1 - \bar P_2 -\ldots - \bar P_k$ in $\TX$
such that every two adjacent points belong to the same $(\km{m})$ for one of the projections, 


%

\refstepcounter{equation}
\subsection{Proof of Proposition~\ref{theorem:p-t-g-1}}\label{sec:together}
Let a variety $\TX$ be as in~(\ref{T_X}). For each $i$ we have a fibration with the quasiaffine base and fiber being $(\km{m_i}\times T^{n-2}$, see~(\ref{Diagr}).

\begin{df}
For the points $C_1,\ldots,C_r$ in the base of projection~(\ref{Diagr}), let $\Stab_{C_1,\ldots,C_r}$ be the subgroup in $\SAut(\TX)$ preserving all the fibers of the projection and  fixing pointwise the fibers above $C_1,\ldots,C_r$.
\end{df}

Now we need some technique concerning locally nilpotent derivations (LNDs). To lift automorphisms, we need to extend an LND on $\km{m}$ to a LND on $\TX$ (i.e. to a locally nilpotent derivation of the algebra $\kk[\TX]$). More precisely, suppose that we chose a fiber of form $\km{m}$ of the projection~(\ref{Diagr}) and some other fibers that we want to fix. We can project all these subvarieties 
to~$\bar Y_i$ and then take a regular function on $\bar Y_i$ which equals~1 at the projection of the first fiber and 0 in the projections of other fibers. 
If we multiply the LND by this function (obviously belonging to the kernel of the derivation) and trivially extend it to the toric factor, we will obtain a rational derivation well-defined on an open subset~$U$ of~$\TX$ corresponding to the smooth locus of the corresponding $\pi_i$. Now we can take a regular function on~$\TX$ (lifted from a regular function on $\bar Y_i$) such that its zero locus contains the singular locus of the projection, multiply the derivation by some power of this function and obtain a regular LND on~$\TX$.

For a given $m$-tuple of points $P_1,P_2,\ldots,P_m$, we need the following lemma:
\begin{lem}
In the notation as above~(\ref{Diagr}), let $C_0$ be a point on a base such that the fiber over this point is general, and $P_1,\ldots,P_s$ be 
some points from this fiber with different projections to~$\bar Y_i$. Let also $C_1,\ldots,C_r$ be some other points of the base. Then the subgroup 
$Stab_{C_1,\ldots,C_r}$ acts infinitely transitively on the fiber over $C_0$, i.e. can map $P_1,\ldots,P_s$ in any other subset in the same fiber.
\end{lem}
\begin{proof}
By~\cite[Theorem 0.2(3)]{AKZ}, the fiber is infinitely transitive. For every one-parameter unipotent subgroup of automorphisms on this fiber, we can lift it to $\TX$, fixing pointwise a given finite collection of fibers, see above. 
\end{proof}

Now it remains to prove infinite transitivity for $\TX$. There are two ways to show it.

\subsubsection{Way 1}
\begin{lem}
For $m+1$ points $P_1, P'_1, P_2,P_3,\ldots,P_m$ projecting to the chosen above open subset in~$X$, there exists an automorphism 
mapping $P_1$ to $P'_1$ and preserving all the other points. 
\end{lem}
\begin{proof}
There always exists a small automorphism which moves the initial set to a set where for all~$i$ all the $\km{m_i}$-coordinates  of the given points are different. Let us connect the projections of $P_1$ and $P'_1$ by a chain of smooth $\p^1_i$-curves in $X$. We denote by $Q_1,\ldots, Q_s$ the intersection points of these curves, $Q_i\in X$, $Q_1=\delta(P_1)$, $Q_s=\delta(P'_1)$. For $i=2,\ldots,(s-1)$ we take some lifts $R_i\in\TX$ of these points in such a way that all their $\km{m_i}$-coordinates do not coincide with the corresponding coordinates of the previous points. Let $R_1:=P_1$ and $R_s=P'_1$. For every~$i$, $1\leqslant i\leqslant (s-1)$, we want to map $R_i$ to $R_{i+1}$ by an automorphism of~$\TX$ preserving all the other points in the given set. We may assume that $R_i$ and $R_{i+1}$ belong to one two-dimensional fiber of form $\km{m}$ of one of the projections to $\bar Y_j \times T_j^{n-2}$ (the toric fibration is generated by $L_i$s, and we can densify the sequence of $R_i$s if needed to change only one $L_i$-direction at every step to fulfill this condition). Every such two-dimensional fiber is inf-transitive. Now we need to lift the corresponding automorphism to $\TX$. We need two following observations. First, if we are lifting a curve with respect to the projection $\pi_i$, then the resulting automorphism is well defined over the singular fibers and is trivial there. Second, all the two-dimensional fibers belonging to the same fiber of $\varphi_i$ move together, and if several $P_j$ belong to the same fiber as the $R_i$ which we are moving, then we use that their projections to the two-dimensional fiber are different and also different from the projection of~$R_{i+1}$, and we use inf-transitivity (not only 1-transitivity) of the corresponding fiber. Now for every~$i$ we lift the corresponding automorphism of the $2$-dimensional fiber to an automorphism of~$\TX$ from the corresponding subgroup~$\Stab$ fixing the points from the other fibers of~$\delta$, and multiply all these automorphisms. It does not change $P_2,\ldots,P_m$ and maps $P_1$ to $P'_1$. This ends the proof.   
\end{proof}

Now infinite transitivity follows easily: to map $P_1,P_2,\ldots,P_m$ to $Q_1,Q_2,\ldots$, $Q_m$, we map $P_1$ to $Q_1$ fixing $P_2,\ldots,P_m,Q_2,\ldots,Q_m$, etc. .

\subsubsection{Way 2}
The other way to finish the proof is as follows. It is enough to show 1-transitivity while fixing some other points of a given finite set. Let us consider automorphisms of bounded degree fixing $P_2,\ldots,P_m$ and the orbits of $P_1$ and $P'_1$ under this group. Clearly, by flexibility every orbit is 
an open subset in~$\TX$, and every two dominant subsets should have a nonempty intersection. So there is a common point, which means that $P_1$ can be mapped to $P'_1$ by a subgroup in $\SAut(\TX)$ fixing $P_2,\ldots,P_m$.

\begin{oss}
\label{remark:converse} Conversely, in view of Theorem~\ref{theorem:t-g-1},
given a b-inf. trans. variety $X$ there exist $\dim X$
cancellations of $X$. Indeed, for general point
$\zeta\in X$
we can find $\dim X$ tangent vectors spanning $T_{X, \zeta}$,
such that each vector generates a copy of $\mathbb{G}_a =: G_i \subseteq\saut(X)$, $1 \leq i \leq n$. Let $\mathfrak{G}\subseteq\saut(X)$ be the subgroup generated by
the groups $G_2, \ldots, G_n$. Then we have $X \approx G_1 \times  \mathfrak{G} \cdot\zeta$.
\end{oss}

\smallskip

\refstepcounter{equation}
\subsection{Increasing the rank of the corresponding subgroup in $H_2$}\label{incr-t}
We want to treat here the case when $\rk \langle \p^1_1, \ldots, \p^1_n \rangle$ is $t$, $t<n$, as of a subgroup 
in $H_2(X)$.

\begin{oss}
Here we precise the ancient construction of $\TX$. Indeed, if the rank is not maximal, then 
the toric bundle contains a trivial part, and
we need to get rid of it. One way is to change it with the trivial vector bundle part. However here we give another construction which uses stabilization. 
\end{oss}

\begin{lem}\label{rank-sg}
There is a stabilization $X'$ of $X$ such that $\dim X' - t(X')< (n-t)$.
\end{lem}
\begin{proof}
We assume that the cycles $\p^1_i$ are dependent and in particular that 
an integer multiple of $\p^1_n$ is contained in the envelope of $\p^1_i$, $i < n$, on $X$. 
There is a natural projection $p_{n,n+1}: X\times \p^1\to Y_n$ with a generic fiber
$\p^1\times \p^1$. 
Let us take $\p^1\times \p^1$ and blow it up at $3$ points.
Thus we will have $\p^2$ with $5$ blown up points. 
For every $4$ points there is a pencil of conics passing
through four points. Indeed, if we fix two smooth conics, 
there is a pencil of conics passing through the intersection.
So on $\p^1\times \p^1$ blown up at three points we can choose
two different $4$-tuples of points on $\p^2$ and define two
projections $\bar\pi_i : Bl_{Q_1,Q_2,Q_3}(\p^1\times \p^1)\to \p^1$.
Now we can extend them to $\p^1\times \p^1\times B$ by blowing
up three constant sections and similarly extend projections.
The projections $\bar\pi_1$, $\bar\pi_2$ provide cancellations with new $\bar \p^1_i$, $i=1,2$, 
independent with generic $\p^1_j$, $j\neq 1,2$ on the blown up $X$. 
We denote the resulting variety by $\tilde X$, it is a smooth model of $X\times \p^1$.
Here the rank $t^1 = t+2$.
\end{proof}

\refstepcounter{equation}
\subsection{Increasing the rank of the matrix~$M$}\label{sec:incr-rank}
For a variety $X$ with a given set of cancellations and corresponding bundles, we constructed~(\ref{Matrix}) a matrix~$M_n$ of restrictions. To prove 
birational stable infinite transitivity, we need the rank of this matrix to be full. The aim of this section is to prove the following lemma.


\begin{lem} \label{rank-m}
Let the rank of the subgroup generated by $\p^1_i$ in $H_2(X,Z)$ be~$n$, and let matrix $M_n=M(X)=(m_{i,j})$ be as in~(\ref{Matrix}) and its rank be $s < n= \dim X$.  
Then there exists a birational model $\tilde X$ for $X\times \p^1$ with
$n+1$ projections corresponding to cancellations and a family
$H^1_i$, $i=1,2,\ldots,(n+1)$, such that $s_1 \geq n+2$ for the new matrix $M(\tilde X)$.
\end{lem}

\begin{proof}
By Lemma~\ref{rank-sg}, we may assume that all the classes $[\p^1_1]$, $[\p^1_2], \ldots$, $[\p^1_n]$ are independent in $H_2(X)$. 
If $s < n= \dim X$, then due to Hodge duality there is a divisor with nonzero positive pairings with all the fibers $\p^1_i$, i.e. there is an ample divisor $H_{n+1}$ on $X$
such that it defines an element in the lattice $\ZZ^n$ which is not contained
in $M (\ZZ(H_1,H_2,\ldots,H_n))$ (here we identify $H_i$ with the elements of the standard basis in~$\ZZ^n$). 
Let us define in this case $\pi_i^1 : X\times \p^1\to (Y_i\times \p^1)$; take $H_i^1= H_i + \hol_{\p^1}(n_i)$ for some positive numbers $n_i$; 
$\pi_{n+1} : X\times \p^1\to X$ the trivial projection; and  $H_{n+1}$ chosen above.
Then if the restriction of $H_{n+1}$ on $\p^1_i$ is $\hol(t_i)$, 
the new matrix $M(\tilde X)$ is as follows:
$$
M(\tilde X)=\qmatrix{ &&& n_1 \\ &M(X)&& \vdots \\ &&& n_n\\ t_1 &\ldots &t_n&0}. 
$$
Note that all the diagonal elements $m_{i,i}= 0$. 
The matrix $M(\tilde X)$ in this case has rank $s+2$ for
some choice of $n_i$. 
Indeed, the last row of $M(\tilde X)$ is
independent with other rows by the assumption on $H_{n+1}$. Now we can add $n_i$ in such a way that 
the rank of $M(\tilde X)$ will be $(\rk M+1)+1$ (if $\rk M < n$).
Hence $\rk M_1= s+2$ in this case. 
\end{proof}

\begin{cor}
In finite number of steps (not more than $2n$), using Lemmas~\ref{rank-sg} and \ref{rank-m}, we obtain
a model $\tilde X$ of $X\times \p^r$ with $\rk M (\tilde X)= \dim (X\times \p^r)=\dim(\tilde X)$.
\end{cor}

\section{Examples}
\label{section:exs}

Here we collect several examples and properties of (stably) b-inf.
trans. varieties.

\refstepcounter{equation}
\subsection{Quotients}

Let us start with the projective space $\p^n$, $n \geq 2$, and a
finite group $G\subset PGL_{n+1}(\fie)$. Notice that the quotient
$\p^n/G$ is stably b-inf. trans. Indeed, let us replace $G$ by its
finite central extension $\tilde{G}$ acting linearly on $V :=
\fie^{n+1}$, so that $V/\tilde{G} \approx \p^n/G \times \p^1$.
Further, form the product $V \times V$ with the diagonal
$\tilde{G}$-action, and take the quotient $V' := (V \times
V)/\tilde{G}$. Then, projecting on the first factor we get $V'
\approx V \times V/\tilde{G}$, and similarly for the second
factor. This implies that $V'$ admits $2n+2$ cancellations (cf.
Theorem~\ref{theorem:t-a-1}). Hence $V'$ is stably b-inf. trans.
by Theorem~\ref{theorem:t-g-1}. The argument just used can be
summarized as follows:

\begin{lem}
\label{theorem:rel-con} Let $X \to S$ be a $\p^m$-fibration for
some $m\in\na$. Then the product $X \times_S X \approx X \times
\fie^m$ admits $2m$ algebraically independent cancellations over
$S$.
\end{lem}

\begin{proof}
Note that $X \times_S X$ has two projections (left and right) onto
$X$, both having a section (the diagonal $\Delta_X \subset
X\times_S X$), hence the corresponding $\p^m$-fibrations are
birational (over $S$) to $X \times \fie^m$. This gives $2m$
algebraically independent cancellations over $S$.
\end{proof}

\begin{cor}
\label{theorem:many-can-cro} Assume that $X$ carries a collection
of distinct birational structures of $\p^{m_i}$-bundles, $\pi_i: X
\to S_i$, with the condition that the tangent spaces of generic
fibers of $\pi_i$ span the tangent space of $X$ at the generic
point. Then $X$ is stably b-inf. trans.
\end{cor}

\begin{proof}
Indeed, after multiplying by the maximum of $m_i$ we may assume
that all $\p^{m_i}$-bundles provide with at least $2m_i$ different
cancellations (see Lemma~\ref{theorem:rel-con}). We can now apply
Theorem~\ref{theorem:t-g-1}.
\end{proof}

\begin{oss}
It seems plausible that
given an inf. trans. variety $X$ and 
a finite group $G\subset\aut(X)$, variety $X/G$
is stably b-inf. trans. (though the proof of this
fact requires a finer understanding of the group
$\saut(X)$).
At this
stage, note also that if $G$ is cyclic, then there exists a
$G$-fixed point on $X$. Indeed, since $X$ is unirational (cf.
Section~\ref{section:intro}), it has trivial algebraic fundamental
group $\pi^{\text{alg}}_1(X)$ (see \cite{kollar-rc}). Then, if the
$G$-action is free on $X$, we get $G \subset
\pi^{\text{alg}}_1(X/G) = \{1\}$ for $X/G$ smooth unirational, a
contradiction. This fixed-point-non-freeness property of $X$
relates $X$ to homogeneous spaces, and it would be interesting to
investigate whether this is indeed the fact, i.e., in particular,
does $X$, after stabilization and passing to birational model,
admit a uniformization which is a genuine (finite dimensional)
algebraic group?\footnote{This question was suggested by
J.-L.\,Colliot-Th\'el\`ene in connection with
Conjecture~\ref{conj:c-main}. However, there are  reasons to doubt
the positive answer, since, for example, it would imply that $X$
is (stably) birationally isomorphic to $G/H$, where both $G, H$
are (finite dimensional) reductive algebraic groups. Even more, up
to stable birational equivalence we may assume that $X = G'/H'$,
where $H'$ is a finite group and $G'$ is the product of a general
linear group, Spin groups and exceptional Lie groups. The latter
implies, among other things, that there are only countably many
stable birational equivalence classes of unirational varieties,
but we could not develop a rigorous argument to bring this to
contradiction.}
\end{oss}

\refstepcounter{equation}
\subsection{Cubic hypersurfaces}

Let $X_3 \subset \p^{n + 1}$, $n \geq 2$, be a smooth cubic. Then

\begin{prop}
\label{theorem:cubic}
$X_3$ is stably b-inf. trans.
\end{prop}

\begin{proof}
Smooth cubic $X_3$ contains a two-dimensional family of lines which span $\p^4$. Let $L \subset X_3$ be a line and $\pi : X_3 \dashrightarrow \p^{n
- 1}$ the linear projection from $L$. Let us resolve the
indeterminacies of $\pi$ by blowing up $X_3$ at $L$. We arrive at
a smooth variety $X_L$ together with a morphism $\pi_L : X_L \to
\p^{n - 1}$ whose general fiber is $\p^1$ ($\simeq$ a conic in
$\p^2$). Varying $L \subset X_3$, we then apply
Lemma~\ref{theorem:rel-con} and
Corollary~\ref{theorem:many-can-cro} to get that $X_3$ is stably
b-inf. trans.
\end{proof}

\refstepcounter{equation}
\subsection{Quartic hypersurfaces}

Let $X_4\subset\p^n$, $n \geq 4$, be a quartic hypersurface with a line $L\subset X_4$ of double singularities. Then

\begin{prop}
\label{theorem:sing-quart}
$X_4$ is stably b-inf. trans.
\end{prop}

\begin{proof}
Consider the cone $\mathfrak{X}_4\subset\p^{n+1}$ over $X_4$. Then
$\mathfrak{X}_4$ contains a plane $\Pi$ of double singularities. Pick
up a (generic) line $L'\subset\Pi$ and consider the linear
projection $\mathfrak{X}_4\dashrightarrow \p^{n-1}$ from $L'$. This
induces a conic bundle structure on $\mathfrak{X}_4$, similarly as in
the proof of Proposition~\ref{theorem:cubic}, and varying $L'$ in
$\Pi$ as above we obtain that $\mathfrak{X}_4$ is stably b-inf. trans.
Then, since  $\mathfrak{X}_4 \approx X_4 \times \fie$,
Proposition~\ref{theorem:sing-quart} follows.
\end{proof}

\refstepcounter{equation}
\subsection{Complete intersections}

Let $X_{2\cdot 2\cdot 2} \subset \p^6$ be the smooth complete intersection of three quadrics. Then

\begin{prop}
\label{theorem:3-quadrics-int}
$X_{2\cdot 2\cdot 2}$ is stably b-inf. trans.
\end{prop}

\begin{proof}
The threefold $X_{2\cdot 2\cdot 2}$ contains at least a one-dimensional family of lines. Let $L \subset X_{2\cdot 2\cdot 2}$ be a line and $X_L \to X_{2\cdot 2\cdot 2}$ the blowup of $L$.
Then the threefold $X_L$ carries the structure of a conic bundle (see \cite[Ch.\,10, Example 10.1.2,\,(ii)]{isk-prok}).
Now, varying $L$ and applying the same arguments as in the proof of Proposition~\ref{theorem:cubic},
we obtain that $X_{2\cdot 2\cdot 2}$ is stably b-inf. trans.
\end{proof}

\begin{oss}
\label{remark:other-ci} Fix $n, r \in\na$, $n \gg r$, and a
sequence of integers $0<d_1\leq \ldots \leq d_m$, $m \geq 2$. Let
us assume that $(n-r)(r+1)\geq \summ_{i=1}^m{d_i+r\choose r}$.
Consider the complete intersection $X:=H_1\cap\ldots\cap H_m$ of
hypersurfaces $H_i\subset\p^n$ of degree $d_i$. Then it follows
from the arguments in \cite{kollar-and-co} that $X$ contains a
positive dimensional family of linear subspaces $\simeq\p^r$.
Moreover, $X$ is unirational, provided $X$ is generic. It would be
interesting to adopt the arguments from the proofs of
Propositions~\ref{theorem:cubic}, \ref{theorem:sing-quart} and
\ref{theorem:3-quadrics-int} to this more general setting in order
to prove that $X$ is stably b-inf. trans.
\end{oss}

\begin{oss}
Propositions~\ref{theorem:cubic}, \ref{theorem:sing-quart} and
\ref{theorem:3-quadrics-int} (cf. Remark~\ref{remark:other-ci})
provide an alternative method of proving unirationality of
complete intersections  (see \cite[Ch.\,10]{isk-prok} for
recollection of classical arguments). Note also that (generic)
$X_{2\cdot 2\cdot 2}$ is non-rational (see for example
\cite{tyurin}), and (non-)rationality of the most of other
complete intersections considered above is not known. At the same
time, verifying stable b-inf. trans. property of other
(non-rational) Fano manifolds (cf. \cite[Ch.\,10, Examples
10.1.3,\,(ii),\,(iii),\,(iv)]{isk-prok}) is out of reach for our
techniques at the moment.
\end{oss}

\end{document}